\documentclass[11pt]{article}
\usepackage[cp1251]{inputenc}
\usepackage{amsmath,amsthm,mathrsfs}
\usepackage[english]{babel}

\usepackage{amsfonts,amstext,amssymb,verbatim,epsfig}%,psfrag}
\usepackage{pgf,tikz}
\usepackage{float}
\usetikzlibrary{arrows}
\usepackage{hyperref}
\usepackage{cite}
\usepackage{epstopdf}
\usepackage{color}
\usepackage{transparent}
\usepackage{subfigure}

\usepackage{xcolor}
\usepackage[normalem]{ulem}

\hypersetup{colorlinks=true,pdfborder=000,  citecolor  = blue, citebordercolor= {magenta}}
\hypersetup{linkcolor=black}
\makeatletter
\let\reftagform@=\tagform@
\def\tagform@#1{\maketag@@@{(\ignorespaces\textcolor{purple}{#1}\unskip\@@italiccorr)}}
\renewcommand{\eqref}[1]{\textup{\reftagform@{\ref{#1}}}}
\makeatother

\makeatletter
\DeclareUrlCommand\ULurl@@{%
  \def\UrlLeft{\uline\bgroup}%
  \def\UrlRight{\egroup}}
\def\ULurl@#1{\hyper@linkurl{\ULurl@@{#1}}{#1}}
\DeclareRobustCommand*\ULurl{\hyper@normalise\ULurl@}
\makeatother

\def\lessim{\ \lower4pt\hbox{$
		\buildrel{\displaystyle <}\over\sim$}\ }
\def\gessim{\ \lower4pt\hbox{$\buildrel{\displaystyle >}
		\over\sim$}\ }

\def\la{\langle}
\def\ra{\rangle}

\newcommand{\e}{\mathbb{E}}
\newcommand{\p}{\mathbb{P}}

\newcommand{\indi}{\ensuremath{\boldsymbol 1}}

\newtheorem{lemma}{\bf Lemma}

\newtheorem{theorem}{\bf Theorem}

\newtheorem{remark}{\bf Remark}
\newtheorem{example}{\bf Example}

\newtheorem{proposition}{\bf Proposition}

\newenvironment{Proof of lemma}{\noindent{\bf Proof of Lemma}}{\hfill$\Box$\newline}
\newenvironment{Proof of theorem}{\noindent{\bf Proof of Theorem}}{\hfill{\footnotesize${\square}$}\newline}
\newenvironment{Proof of theorems}{\noindent{\bf Proof of Theorems}}{\hfill$\Box$\newline}
\newenvironment{Proof of proposition}{\noindent{\bf Proof of Proposition}}{\hfill$\Box$\newline}
\newenvironment{Proof of propositions}{\noindent{\bf Proof of Propositions}}{\hfill$\Box$\newline}
\newenvironment{Proof of exercise}{\noindent{\it Proof of Exercise:}}{\hfill$\Box$}
\begin{document}

\title{On concentration properties of disordered Hamiltonians}

\author{Antonio Auffinger \thanks{Department of Mathematics. Email: auffing@math.northwestern.edu} 
        \\ \small{Northwestern University} 
	\and Wei-Kuo Chen  \thanks{School of Mathematics. Email: wkchen@umn.edu}
	\\ \small{University of Minnesota}
	}\maketitle


\begin{abstract}
We present an elementary approach to concentration of disordered Hamiltonians. Assuming differentiability of the limiting free energy $F$ with respect to the inverse temperature $\beta$, we show that the Hamiltonian concentrates around the energy level $F'(\beta)$  under the free energy and Gibbs average.
\end{abstract}

\maketitle

\section{Introduction and main results} 	
	
	The aim of this short note is to establish a general principle on the concentration of disordered Hamiltonians, arising from various contexts of statistical mechanics models, assuming the differentiability of the limiting free energy. To begin with, for each $N\geq 1$, let $(\Sigma_N,\mathcal{F}_N)$ be a measurable space and $\nu_N$ be a random probability measure on this space. We call $\Sigma_N$ a configuration space. A Hamiltonian $H_N$ is a stochastic process indexed by $\Sigma_N$ with
	\begin{align}\label{ass2}
	\e \int_{\Sigma_N}\exp \bigl(\beta |H_N(\sigma)|\bigr)\nu_N(d\sigma)<\infty,\,\,\forall \beta>0,
	\end{align} 
	where $\e$ is the expectation with respect to the randomness of $H_N$ and $\nu_N.$
	For a given (inverse) temperature $\beta\geq 0,$ 
	the free energy and Gibbs measure associated to $H_{N}$ are defined respectively as
		\begin{align*}
		F_N(\beta)&=\frac{1}{N}\log Z_{N}(\beta)
		\end{align*}
		and
    		\begin{align*}
    		G_{N,\beta}(d\sigma)&=\frac{\exp\bigl(\beta H_N(\sigma)\bigr)\nu_N(d\sigma)}{Z_N(\beta)},
    		\end{align*}
where $$Z_N(\beta):=\int_{\Sigma_N}\exp \big( \beta H_N(\sigma)\big)\nu_N(d\sigma)$$ is called the partition function. Here the assumption \eqref{ass2} justifies the definiteness of $F_N$ and $G_{N,\beta}.$ We emphasize that $F_N$ and $G_{N,\beta}$ are random objects depending on $H_N$ and $\nu_N.$ In particular, from H\"{o}lder's inequality, $F_N$ is convex in $\beta.$ 
Denote by 
$\left<\cdot\right>_\beta$ the expectation (Gibbs average) with respect to the Gibbs measure $G_{N,\beta}$, that is, for an integrable function $\Psi : \Sigma_{N} \to \mathbb R$, write 
\[
\left<\Psi \right>_\beta = \frac{1}{Z_N(\beta)}\int_{\Sigma_{N}} \Psi(\sigma) \exp \bigl(\beta H_{N}(\sigma)\bigr) \nu_{N}(d\sigma).
\]
       
       We assume that the following condition is in force throughout the remainder of the paper. Let $\delta>0$ and $\beta>0$ be fixed. Suppose that there exists a nonrandom function $F:(\beta-\delta,\beta+\delta)\rightarrow\mathbb{R}$ such that for any $\beta'\in(\beta-\delta,\beta+\delta)$, 
		\begin{align}
		\label{ass1}
			\lim_{N\rightarrow\infty}F_N(\beta')=F(\beta'), \,\,a.s.
		\end{align}
    		The assumption of $F$ being nonrandom appears in many examples of disordered systems as a consequence of concentration of measure, where 
in most cases 
\[ F(\beta) = \lim_{N\to \infty} \mathbb E F_{N}(\beta);
\] 		
see Example \ref{ex1} below.
    		First we show that the Hamiltonian is concentrated around a fixed energy level under the free energy if we assume that $F$ is differentiable at $\beta$.

    		\begin{theorem}\label{thm9} Assume that \eqref{ass1} holds and $F$ is differentiable at $\beta.$ For any $\lambda>0,$ define
    			\begin{align*}
    			\delta_N^+(\beta,\lambda)&=\frac{F_N(\beta+\lambda)-F_N(\beta)}{\lambda}-F'(\beta),\\
    			\delta_N^-(\beta,\lambda)&=\frac{F_N(\beta-\lambda)-F_N(\beta)}{\lambda}+F'(\beta).
    			\end{align*}
    			For any $c,c'>0$ with $c+c'<1$, let
    			\begin{align*}
    			\varepsilon_N&={N^{-c'}}+\max\bigl(\delta_N^+(\beta,N^{-c}),\delta_N^-(\beta,N^{-c}),0\bigr)
    			\end{align*}
    			and set
    			\begin{align*}
    			C_N&=\Bigl\{\sigma\in\Sigma_N:\Bigl|\frac{H_N(\sigma)}{N}-F'(\beta)\Bigr|\leq \varepsilon_N\Bigr\}
    			\end{align*}
    			for $N\geq 1.$ We have that
    			
    			\begin{enumerate}
    				\item[$(i)$] $\lim_{N\rightarrow\infty}\varepsilon_N=0$ a.s.
    				\item[$(ii)$] $G_{N,\beta}(C_N^c)\leq 2e^{-N^{1-(c+c')}},$ where $C_N^c$ is the complement of $C_N.$ 
    				\item[$(iii)$] For any $N\geq \frac{\log\log 2}{1-(c+c')},$ the following inequality holds,
    				$$\Bigl|\frac{1}{N}\log \int_{\Sigma_N}\indi_{C_N}(\sigma)\exp\big(\beta H_N(\sigma)\big)\nu_N(d\sigma)-F_N(\beta)\Bigr|\leq \frac{2e^{-N^{1-(c+c')}}}{N}.$$
    			where $\indi_{C_N}$ is the indicator function on $C_N.$
    			\end{enumerate}

    		\end{theorem}
    	
    	Three remarks are in position.
    	
    	\begin{remark}\label{rmk1}\rm	
Items $(i)$ and $(iii)$ say that in the computation of the limiting free energy, only the configurations associated to energies near the energy level $F'(\beta)$ with an error estimate $\varepsilon_N$ would have nontrivial contributions. 
%In physics, the map between temperature parameter $\beta$ and energy level $F'(\beta)$ provided in Theorem \ref{thm9} is referred as the energy-temperature relation in the thermodynamic equivalence of the canonical and microcanonical ensembles. For more in this direction, see \cite[Section III]{Touchette} and the references therein.
As an immediate consequence of $(iii)$, one can compute the entropy  of the configurations at the energy level $F'(\beta)$ by
    		\begin{align}
    		\lim_{N\rightarrow\infty}\frac{1}{N}\log \nu_N(C_N)=F(\beta)-\beta F'(\beta).
    		\label{eq:entropy}
    		\end{align}  
    		\end{remark}
    		
    		\begin{remark}\label{rmk2}
    			\rm From \eqref{eq:entropy} and $(ii)$, we can deduce that the conditional Gibbs measure on $C_N$ at temperature $\beta$ is equivalent, at exponential scale, to the probability measure $\nu_N$ conditioning on $C_N,$ that is, for any sequence of sets $(A_N)$ with $A_N\subset\Sigma_N$  and $A_N\cap C_N\neq \emptyset$ for all $N\geq 1$,  
    			\begin{align*}
    			\lim_{N\rightarrow\infty}\frac{1}{N}\log \frac{G_{N,\beta}(A_N|C_N)}{\nu_N(A_N|C_N)}=0 \quad a.s.,
    			\end{align*}
    			where $G_{N,\beta}(A|C)$ and $\nu_N(A|C)$ are the conditional probabilities of event $A$ given $C$ with respect to $G_{N,\beta}$ and $\nu_N,$ respectively. Indeed, this can be obtained by 
    			\begin{align*}
    			\frac{1}{N}\log G_{N,\beta}(A_N| C_N)&=\frac{1}{N}\log G_{N,\beta}(A_N\cap C_N)+o(1)\\
    			&=\frac{1}{N}\log \nu_N(A_N\cap C_N)+\beta F'(\beta)-\frac{1}{N}\log Z_N+o(1)\\
    			&=\frac{1}{N}\log \nu_N(A_N\cap C_N)+\beta F'(\beta)-F(\beta)+o(1)\\
    			&=\frac{1}{N}\log \nu_N(A_N\cap C_N)-\frac{1}{N}\log \nu_N(C_N)+o(1)\\
    			&=\frac{1}{N}\log \nu_N(A_N| C_N)+o(1),
    			\end{align*}
    			where the first equality used $(ii)$ and $o(1)$ is a function (that may change from line to line) such that $\lim_{N\rightarrow \infty}o(1)=0.$
    		\end{remark}
    		
    			\begin{remark}
    				\label{rm3}
    				\rm 	Naturally, one would wonder what will happen if $F$ is not differentiable at $\beta.$ In this case, it can be shown, following the argument of Theorem \ref{thm9}, that the Hamiltonian stays inside the interval $(D_-F(\beta)-\varepsilon_N',D_+F(\beta)+\varepsilon_N')$ under the free energy, where $D_{\pm}F(\beta)$ are the right and left derivatives of $F$  and the quantity $\varepsilon_N'$ converges to zero. It would be of great interest to construct examples with their Hamiltonians being supported at more than one point in the interval $(D_-F(\beta),D_+F(\beta)).$ 
    			\end{remark}

    		Note that a classical result in convex analysis \cite{Ro} states that the derivative of a differentiable convex function on an open interval must be continuous on that interval. Thus, if one  further assumes that $F$ is differentiable on an open interval $J\subseteq (\beta-\delta,\beta+\delta)$, then $F'$ is automatically a continuous function on $J$.  Conversely, the following theorem states that if there exists a continuous function $E$ on some open interval $J\subset(\beta-\delta,\beta+\delta)$ such that the Hamiltonian is concentrated around $E$ under the free energy for all temperature in $J$, then  $F$ must be differentiable in $J$. 
    		\begin{theorem}\label{prop2}
    			If there exists an open interval $J\subseteq(\beta-\delta,\beta+\delta)$ and a continuous function ${E}:J\to \mathbb{R}$ such that for any $\beta\in J$ and $\varepsilon>0,$
    			\begin{align*}
    			\lim_{N\rightarrow\infty}\frac{1}{N}\log \int_{\Sigma_N}\indi_{\bigl\{\bigl|\frac{H_N(\sigma)}{N}-E(\beta)\bigr|\leq \varepsilon\bigr\}}(\sigma)\exp\big(\beta H_N(\beta)\big)\nu_N(d\sigma)=F(\beta), \, \, a.s.
    			\end{align*}
    			Then $F$ is differentiable on $J$ and $E=F'$ on $J.$
    		\end{theorem}

    		We now turn our attention to the concentration of the Hamiltonian under the Gibbs average. This property has been derived under various assumptions and plays an essential role in statistical mechanics. Our main result below establishes concentration of the Hamiltonian under the Gibbs measure a.s.
   		
    		\begin{theorem}\label{prop3}
    			Assume that \eqref{ass1} holds and $F$ is differentiable at $\beta.$ Then
    			\begin{align*}
    			\lim_{N\rightarrow\infty}\Bigl\la\Bigl|\frac{H_N(\sigma)}{N}-F'(\beta)\Bigr|\Bigr\ra_\beta=0,\,\,a.s.
    			\end{align*}	
    		\end{theorem}
    		
    		A similar result was obtained in Panchenko \cite{P10}, where he presented an approach, different from the one adopted in this paper, to establish
    		\begin{align}\label{eq1}
    		\lim_{N\rightarrow\infty}\e\Bigl\la\Bigl|\frac{H_N(\sigma)}{N}-\e\Bigl\la\frac{H_N(\sigma)}{N}\Bigr\ra_\beta\Bigr|\Bigr\ra_\beta=0
    		\end{align}
    		under the assumption that $\lim_{N\rightarrow\infty}\e |F_N-\e F_N|=0$, $\lim_{N\rightarrow\infty}\e F_N=F$ in $(\beta-\delta,\beta+\delta)$, and $F$ is differentiable at $\beta.$ 
    		Note that Theorem \ref{prop3} readily implies \eqref{eq1}. Although the assumption \eqref{ass1} in Theorem \ref{prop3} is stronger than Panchenko's setting described above, our argument of Theorem \ref{prop3} indeed allows to obtain the same statement as Panchenko \cite{P10}. In what follows, we give an example to demonstrate an application of \eqref{eq1} in the study of mean-field spin glasses \cite{MPV}.

       		\smallskip
       		
       		\begin{example}[Sherrington-Kirkpatrick model]\label{ex1}
       			\rm For any $N\geq 1,$ consider the hypercube $\Sigma_N=\{-1,+1\}^N$ and let $\nu_N$ be the uniform probability measure on $\Sigma_N.$ The Sherrington-Kirkpatrick (SK) mean-field spin glass model is defined on $\Sigma_N$ with Hamiltonian,
       			\begin{align*}
       			H_N(\sigma)=\frac{1}{\sqrt{N}}\sum_{i,j=1}^Ng_{ij}\sigma_{i}\sigma_{j},
       			\end{align*}
       			where $g_{ij}$'s are i.i.d. standard normal for all $1\leq i,j\leq N$. It can be easily computed that the covariance of $H_N$ equals
       			\begin{align*}
       			\e H_N(\sigma)H_N(\sigma')=N\Bigl(\frac{1}{N}\sum_{i=1}^N\sigma_i\sigma_i'\Bigr)^2
       			\end{align*}
       			for any $\sigma,\sigma'\in \Sigma_N.$ As the covariance of $H_N$ is of order $N$, the quantity $1/N$ in the definition of the free energy is the right scaling factor. In Guerra-Toninelli \cite{GT}, it was known that the limiting free energy of the SK model converges a.s. and, by the virtue of Gaussian concentration of measure, this limit is a nonrandom function $F$ of $\beta$. Furthermore, it was later established by Talagrand \cite{Tal03} that $F$ admits a variational principle conjectured by Parisi \cite{Par}. One of the important consequences of Parisi's formula guarantees that $F$ is differentiable for all $\beta>0$ (see, e.g., \cite{T11,P13}), so the conclusions of Theorems \ref{thm9} and \ref{prop3} as well as \eqref{eq1} hold.  
       			
       			Denote by $(\sigma^\ell)_{\ell\geq 1}$ i.i.d. samplings from the Gibbs measure and by $\la\cdot\ra_\beta$ the Gibbs expectation with respect to this sequence. Set $R_{\ell,\ell'}=N^{-1}\sum_{i=1}^N\sigma_i^\ell\sigma_i^{\ell'}$ the overlap between $\sigma^\ell,\sigma^{\ell'}$, which measures the degree of similarities between the two configurations. 	
       			Let $n\geq 2$ be fixed. Assume that $\phi$ is a bounded function of the overlaps $(R_{\ell,\ell'})_{1\leq \ell\neq \ell'\leq n}$. Applying \eqref{eq1} yields
       			\begin{align}\label{eq3}
       			\Bigl|\e\Bigl\la\phi\frac{H_N(\sigma^1)}{N}\Bigr\ra_\beta-\e \la\phi\ra_\beta\e\Bigl\la\frac{H_N(\sigma^1)}{N}\Bigr\ra_\beta\Bigr|\leq \|\phi\|_\infty \e\Bigl\la\Bigl|\frac{H_N(\sigma^1)}{N}-\e\Bigl\la\frac{H_N(\sigma^1)}{N}\Bigr\ra_\beta\Bigr|\Bigr\ra_{\beta}\rightarrow 0.
       			\end{align}
       			Recall the Gaussian integration by parts states that if $z$ is a standard normal random variable, then $\e zf(z)=\e f'(z)$ for any absolutely continuous function $f$ with moderate growth. Utilizing this formula, we compute
       			\begin{align*}
       			\e \Bigl\la\phi \frac{H_N(\sigma^1)}{N}\Bigr\ra_\beta&=\e \Bigl\la\phi \sum_{\ell=1}^{n}R_{1,\ell}^2\Bigr\ra_\beta-n\e\bigl\la \phi R_{1,n+1}^2\bigr\ra_\beta,\\
       			 \e \Bigl\la\frac{H_N(\sigma^1)}{N} \Bigr\ra_\beta&=\e \la R_{1,1}^2\ra_\beta-\e\la R_{1,2}^2\ra_\beta.
       			\end{align*}
       		    Plugging these two equations into \eqref{eq3} leads to 
       		    \begin{align}\label{ggi}
       		    \lim_{N\rightarrow\infty}\Bigl(\e \la \phi R_{1,{n+1}}^2\ra_\beta-\frac{1}{n}\e\la \phi\ra_\beta\e\la R_{1,2}^2\ra_\beta-\frac{1}{n}\sum_{\ell=2}^{n}\e \la \phi R_{1,\ell}^2\ra_\beta\Bigr)=0.
       		    \end{align}
       		    This is called the Ghirlanda-Guerra identity \cite{G98} for the SK model. By adding asymptotically vanishing perturbations, one can actually derive the extended Ghirlanda-Guerra identities, where \eqref{ggi} is not only valid  for the second moment of the overlap, but also for any higher moments. These identities contain vital information about the Gibbs measure and ultimately connect to the computation of the limiting free energy. See, for instance, Panchenko \cite{P13} and Talagrand \cite{T11}. We also invite the readers to  check some examples of quenched self-averaging of the Hamiltonian in Auffinger-Chen \cite{AC14}, Chatterjee \cite{C14}, and Chen-Panchenko \cite{CP}.
       		\end{example}
    	
	{\noindent \bf Acknowledgements.} The research of A. A. is partly supported by NSF Grant CAREER DMS-1653552 and NSF Grant DMS-1517894. The research of W.-K. C. is partly supported by NSF Grant DMS-1642207 and Hong Kong Research Grants Council GRF-14302515. Both authors thank the anonymous referees for the careful reading and many valuable suggestions regarding the presentation of the paper. W.-K. C. thanks the hospitality of the institute of mathematics at Academia Sinica and the department of applied mathematics at National Sun Yat-Sen University during his visit in June 2017, where part of the writing of the present paper was completed.

    		\section{Proofs of Theorems \ref{thm9} and \ref{prop2} and Proposition \ref{prop3}}\label{sub4.3}
    	
    		The idea of our proof for Theorem \ref{thm9} is motivated by the standard derivation of the large deviation principle. The added technicality comes from the fact that $(F_N)_{N\geq 1}$ is a sequence of random Laplace transforms rather than deterministic ones. This difficulty will be overcome by applying the almost surely pointwise convergence of the sequence $(F_N)_{N\geq 1}$ stated in the following proposition.  

    	\begin{proposition}\label{thm0}
    		Let $F$ be a continuous function defined on an interval $I \subseteq \mathbb R$. If $$\lim_{N\rightarrow\infty}F_N(\beta)=F(\beta),\,\,a.s.$$
    		for any $\beta$ in a dense subset of $I$, then 
    		\begin{align*}
    		%\label{thm0:eq0}
    		\p\Bigl(\lim_{N\rightarrow\infty}F_N(\beta)=F(\beta),\,\,\forall \beta \in I \Bigr)=1.
    		\end{align*} 
    	\end{proposition}
	To prove Proposition \ref{thm0}, we first need a lemma:
	\begin{lemma}\label{lem0}
    		Let $I\subset \mathbb{R}$ be an open interval. Suppose that $\{f_n\}$ is a sequence of convex functions on $I$ and $f$ is a real-valued function on $I.$ If $\lim_{n\rightarrow\infty}f_n(y)=f(y)$ pointwise on a dense subset $D\subset I$ and $f$ is continuous at $y_0\in \mathbb{R},$ then $\lim_{n\rightarrow\infty}f_n(y_0)=f(y_0)$.
    	\end{lemma}
    	
     	\begin{proof}
     		Let $y_0\in \mathbb{R}$ be fixed. Suppose that $f$ is continuous at $y_0.$ Choose points $a,b,a',b'\in D$ with $a<b<y_0<a'<b'$. By the convexity of $f_n,$ for any $b<x<y<a'$,
     		\begin{align*}
     		\frac{f_n(b)-f_n(a)}{b-a}\leq \frac{f_n(y)-f_n(x)}{y-x}\leq \frac{f_n(b')-f_n(a')}{b'-a'}.
     		\end{align*}
     		Since $\{f_n\}$ converges to $f$ at $a,b,a',b',$ this inequality means that $\{f_n\}$ is uniform Lipschitz on $[b,a']$ for all $n\geq 1$ with Lipschitz constant $M>0.$ Consequently, for any $y\in D\cap [b,a']$,
    		\begin{align*}
    		|f_n(y_0)-f(y_0)|&\leq |f_n(y_0)-f_n(y)|+|f_n(y)-f(y)|+|f(y)-f(y_0)|\\
    		&\leq M|y_0-y|+|f_n(y)-f(y)|+|f(y)-f(y_0)|
    		\end{align*}
    		and passing to limit gives
    		\begin{align*}
    		\limsup_{n\rightarrow\infty}|f_n(y_0)-f(y_0)|\leq M|y_0-y|+|f(y)-f(y_0)|.
    		\end{align*}
    		Since this holds for any $y\in D\cap [b,a']$, letting $y\in D\rightarrow y_0$ finishes our proof.
    	\end{proof}

\begin{proof}[\bf Proof of Proposition \ref{thm0}]
	Pick a countable dense subset $D$ of $I$. Denote by $\Omega({\beta})$ the event that $(F_N(\beta))$ converges. Let $\Omega=\cap_{\beta\in D}\Omega(\beta).$ Note that $\p(\Omega)=1.$ Therefore, on $\Omega,$ $\lim_{N\rightarrow\infty}F_{N}(\beta)=F(\beta)$ for any $\beta\in D.$ Since $D$ is dense and $F$ is continuous everywhere, Lemma \ref{lem0}  implies that $\lim_{N\rightarrow\infty}F_N=F$ on pointwise on $I$ with probability one.
\end{proof}

    		\begin{proof}[\bf Proof of Theorem \ref{thm9}]
    			Let $\lambda_N=N^{-c}.$ To verify $(i)$, suppose that we are on the event,
    			\begin{align*}
    			\bigl\{\lim_{N\rightarrow\infty}F_N(\beta')=F(\beta'),\,\,\forall \beta'\in(\beta-\delta,\beta+\delta)\bigr\}.
    			\end{align*}
    			From the convexity of $F_N$, for any $0<\eta<\delta/2$, as long as $N$ is large enough, we have
    			\begin{align*}
    			\frac{F_N(\beta)-F_N(\beta-\eta)}{\eta}&\leq \frac{F_N(\beta)-F_N(\beta-\lambda_N)}{\lambda_N}\\
    			&\leq \frac{F_N(\beta+\lambda_N)-F_N(\beta)}{\lambda_N}\leq \frac{F_N(\beta+\eta)-F_N(\beta)}{\eta}.
    			\end{align*}
    			Here since the left and right sides converge to 
    			\begin{align*}
    			\frac{F(\beta)-F(\beta-\eta)}{\eta}\,\,\mbox{and}\,\,\frac{F(\beta+\eta)-F(\beta)}{\eta},
    			\end{align*}
    			the differentiability of $F$ shows that these two quantities are equal to each other as $\eta\downarrow 0.$ To sum up,
    			\begin{align}\label{proof:eq1}
    			\lim_{N\rightarrow\infty}\frac{F_N(\beta+\lambda_N)-F_N(\beta)}{\lambda_N}=\lim_{N\rightarrow\infty}\frac{F_N(\beta)-F_N(\beta-\lambda_N)}{\lambda_N}=F'(\beta).
    			\end{align}
    			Note that since $F_N$ is convex, $F$ is continuous on $(\beta-\delta,\beta+\delta).$ From Proposition \ref{thm0} and \eqref{proof:eq1}, $(i)$ follows.
    			
    			The proof of $(ii)$ and $(iii)$ given below is the main novelty of the paper. For any $\varepsilon>0,$ define 
       			\begin{align*}
    			B_N^+(\beta,\varepsilon)&=\frac{1}{N}\log \int_{\Sigma_N}\indi_{A_N^+(\varepsilon)}(\sigma)\exp\bigl(\beta H_N(\sigma)\bigr)\nu_N(d\sigma),\\
    			B_N^-(\beta,\varepsilon)&=\frac{1}{N}\log \int_{\Sigma_N}\indi_{A_N^-(\varepsilon)}(\sigma)\exp\bigl(\beta H_N(\sigma)\bigr)\nu_N(d\sigma),
    			\end{align*}
    			where
    				\begin{align*}
    				A_N^+(\varepsilon)&:=\Bigl\{\sigma\in \Sigma_N:\frac{H_N(\sigma)}{N}-F'(\beta)> \varepsilon\Bigr\},\\
    				A_N^-(\varepsilon)&:=\Bigl\{\sigma\in \Sigma_N:\frac{H_N(\sigma)}{N}-F'(\beta)<- \varepsilon\Bigr\}.
    				\end{align*}
    			Observe that for any $0<\lambda<\delta,$
    			\begin{align}\label{eq1.1}
    			\begin{split}
    			B_N^+(\beta,\varepsilon)&\leq \frac{1}{N}\log \int_{\Sigma_N}\exp\bigl((\beta+\lambda) H_N(\sigma)-\lambda N\bigl(F'(\beta) +\varepsilon\bigr)\bigr)\nu_N(d\sigma)\\
    			&=F_N(\beta+\lambda)-\lambda (\varepsilon+F'(\beta) )
    			\end{split}
    			\end{align}
    			and
    			\begin{align}\label{eq1.2}
    			\begin{split}
    			B_N^-(\beta,\varepsilon)&\leq \frac{1}{N}\log \int_{\Sigma_N}\exp\bigl((\beta -\lambda)H_N(\sigma)+\lambda N\bigl(F'(\beta) -\varepsilon\bigr)\bigr)\nu_N(d\sigma)\\
    			&=F_N(\beta-\lambda)-\lambda (\varepsilon-F'(\beta) ).
    			\end{split}
    			\end{align}
    			To control these inequalities, write
    			\begin{align}\label{eq-1}
    			F_N(\beta\pm\lambda)\mp \lambda F'(\beta) -\lambda \varepsilon&=F_N(\beta)-\lambda(\varepsilon-\delta_N^{\pm}(\beta,\lambda))\leq F_N(\beta)-\lambda(\varepsilon-\gamma_N(\beta,\lambda)),
    			\end{align}  
    			where
    			\begin{align*}
    			\gamma_N(\beta,\lambda)&:=\max\bigl(\delta_N^+(\beta,\lambda),\delta_N^-(\beta,\lambda),0\bigr),\,\,\forall 0<\lambda<\delta.
    			\end{align*}
    			From \eqref{eq1.1}, \eqref{eq1.2}, and \eqref{eq-1}, it follows that
    			\begin{align}
    			\begin{split}
    			\label{thm8:proof:eq1}
    			G_{N,\beta}\bigl(A_N^{\pm}(\varepsilon)\bigr)\leq e^{-N\lambda (\varepsilon-\gamma_N(\beta,\lambda))}.
    			\end{split}
    			\end{align}
    			Consequently, we have
    			\begin{align}\label{eq2}
    			G_{N,\beta}\bigl(A_N^+(\varepsilon)\cup A_N^-(\varepsilon)\bigr)\leq 2e^{-N\lambda (\varepsilon-\gamma_N(\beta, \lambda))}.
    			\end{align}
    			From now on, take $\lambda=\lambda_N$ and $\varepsilon=\varepsilon_N$. Using the bound 
    				\begin{align*}
    				e^{-N\lambda (\varepsilon-\gamma_N(\beta,\lambda))}&\leq e^{-N\cdot N^{-c}\cdot N^{-c'}}<e^{-N^{1-(c+c')}},
    				\end{align*}
    				the inequality \eqref{eq2} gives $(ii).$ On the other hand, observe that 
			\begin{align*}
    			 G_{N,\beta}(C_N) =1- G_{N,\beta}\bigl(A_N^+(\varepsilon)\cup A_N^-(\varepsilon)\bigr) \geq 1- 2e^{-N\lambda (\varepsilon-\gamma_N(\beta, \lambda))}>1-2e^{-N^{1-(c+c')}},
    			\end{align*}
    			from which taking $N^{-1}\log$ on both sides yields that 
    			\begin{align*}
    			%\begin{split}\label{eq0}
    			F_N(\beta)&\geq\frac{1}{N}\log \int_{\Sigma_N}\indi_{C_N}(\sigma)\exp\bigl({\beta H_N(\sigma)}\bigr)\nu_N(d\sigma)\\
    			&\geq F_N(\beta)+\frac{1}{N}\log \bigl(1-2e^{-N\lambda (\varepsilon-\gamma_N(\beta,\lambda))}\bigr)\\
    			&\geq F_N(\beta)+\frac{1}{N}\log \bigl(1-2e^{-N^{1-(c+c')}}\bigr).
    			%\end{split}
    			\end{align*}
    			This implies $(iii)$ and finishes our proof.
    		\end{proof}
    		
    		\begin{remark}
    			\rm From the above proof, under the same assumption and noting $F_N'(\beta)=N^{-1}\bigl\la H_N(\sigma)\bigr\ra_\beta$, one can derive by the same argument to obtain an identical statement as Theorem \ref{thm9} by replacing every $F'(\beta)$ by $F_N'(\beta)$. 
    		\end{remark}
    		
    		\begin{proof}[\bf  Proof of Theorem \ref{prop2}]
    			Let $\beta\in J$ be fixed. For $\varepsilon>0,$ observe that for any $\beta'<\beta$ and $\beta'$ being sufficiently close to $\beta,$ the continuity of $E$ gives
    			\begin{align*}
    			&\frac{1}{N}\log \int_{\Sigma_N}\indi_{\bigl\{\bigl|\frac{H_N(\sigma)}{N}-E(\beta)\bigr|\leq \varepsilon\bigr\}}(\sigma)e^{\beta H_N(\sigma)}\nu_N(d\sigma)\\
    			&\leq \frac{1}{N}\log \int_{\Sigma_N}\indi_{\bigl\{\bigl|\frac{H_N(\sigma)}{N}-E(\beta')\bigr|\leq 2\varepsilon\bigr\}}(\sigma)e^{\beta H_N(\sigma)}\nu_N(d\sigma)\\
    			&\leq \frac{1}{N}\log \int_{\Sigma_N}\indi_{\bigl\{\bigl|\frac{H_N(\sigma)}{N}-E(\beta')\bigr|\leq 2\varepsilon\bigr\}}(\sigma)e^{\beta' H_N(\sigma)}\nu_N(d\sigma)+(\beta-\beta')\bigl({E}(\beta')+2\varepsilon\bigr)
    			\end{align*}
    			and
    				\begin{align*}
    				&\frac{1}{N}\log \int_{\Sigma_N}\indi_{\bigl\{\bigl|\frac{H_N(\sigma)}{N}-E(\beta)\bigr|\leq 2\varepsilon\bigr\}}(\sigma)e^{\beta H_N(\sigma)}\nu_N(d\sigma)\\
    				&\geq \frac{1}{N}\log \int_{\Sigma_N}\indi_{\bigl\{\bigl|\frac{H_N(\sigma)}{N}-E(\beta')\bigr|\leq \varepsilon\bigr\}}(\sigma)e^{\beta H_N(\sigma)}\nu_N(d\sigma)\\
    				&\geq \frac{1}{N}\log \int_{\Sigma_N}\indi_{\bigl\{\bigl|\frac{H_N(\sigma)}{N}-E(\beta')\bigr|\leq \varepsilon\bigr\}}(\sigma)e^{\beta' H_N(\sigma)}\nu_N(d\sigma)+(\beta-\beta')\bigl({E}(\beta')-\varepsilon\bigr).
    				\end{align*}
    			By letting $N\rightarrow \infty$, the given assumption leads to
    			\begin{align*}
    			(\beta-\beta')\bigl({E}(\beta')-\varepsilon\bigr)\leq F(\beta)-F(\beta')\leq(\beta-\beta')\bigl({E}(\beta')+2\varepsilon\bigr),
    			\end{align*}
    			from which and the continuity of $E$, $F$ is left differentiable at $\beta$. Note that one may actually interchange the role of $\beta$ and $\beta'$ in the last inequality to get that for $\beta'>\beta,$
    			\begin{align*}
    			(\beta'-\beta)\bigl({E}(\beta)-\varepsilon\bigr)\leq F(\beta')-F(\beta)\leq(\beta'-\beta)\bigl({E}(\beta)+2\varepsilon\bigr).
    			\end{align*}
    			This leads to the right differentiability of $F$ at $\beta.$ All these together imply that $F$ is differentiable at $\beta$ and $F'(\beta)=E(\beta).$ This finishes our proof.
    		\end{proof}

    		\begin{proof}[\bf Proof of Theorem \ref{prop3}]
    			Recall that \eqref{thm8:proof:eq1} holds for all $\varepsilon>0.$ Integrating against $\varepsilon$ leads to
    			\begin{align}
    			\begin{split}
    			\label{thm7:proof:eq1}
    			\Bigl<\Bigl(\frac{H_N(\sigma)}{N}-F'(\beta)\Bigr)_{\pm}\Bigr>_\beta&\leq \int_0^\infty e^{-N\lambda (\varepsilon-\gamma_N(\beta,\lambda))}d\varepsilon=\frac{1}{N\lambda}e^{N\lambda\gamma_N(\beta,\lambda)}
    			\end{split}
    			\end{align}
    			for all $\lambda\in(0,\delta)$,	where $x_+:=\max(x,0)$ and $x_-:=\max(-x,0)$ for any $x\in\mathbb{R}.$
    			Let $0<\lambda_0<\delta.$ Observe that by the convexity of $F_N,$ for $0<\lambda<\lambda_0,$ we have
    			\begin{align*}
    			\gamma_N(\beta,\lambda)&\leq \gamma_N(\beta,\lambda_0).
    			\end{align*}
    			If $\lambda_0 N\gamma_N(\beta,\lambda_0)\leq 1,$ we take $\lambda=\lambda_0$. If $\lambda_0 N\gamma_N(\beta,\lambda_0)>1,$ then 
    			$1/N\gamma_N(\beta,\lambda_0)\leq \lambda_0$ and we let $\lambda=1/N\gamma_N(\beta,\lambda_0).$ Thus, we conclude from \eqref{thm7:proof:eq1} that
    			\begin{align*}
    			\Bigl<\Bigl(\frac{H_N(\sigma)}{N}-F'(\beta)\Bigr)_{\pm}\Bigr>_\beta&\leq e\max\Bigl(\frac{1}{N\lambda_0},\gamma_N(\beta,\lambda)\Bigr).
    			\end{align*}
    			As a result, for any $0<\lambda_0<\delta,$
    			\begin{align*}
    			\Bigl<\Bigl|\frac{H_N(\sigma)}{N}-F'(\beta)\Bigr|\Bigr>_\beta&\leq \frac{2e}{N\lambda_0}+2e\gamma_N(\beta,\lambda_0)
    			\end{align*}
    			and passing to limit gives
    			\begin{align*}
    			\limsup_{N\rightarrow\infty}\Bigl<\Bigl|\frac{H_N(\sigma)}{N}-F'(\beta)\Bigr|\Bigr>_\beta&\leq 2e\limsup_{N\rightarrow\infty}\gamma_N(\beta,\lambda_0).
    			\end{align*}
    			Here the right-hand side tends to zero as $\lambda_0\downarrow 0$ by a similar argument as \eqref{proof:eq1}. This completes our proof.
    		\end{proof}


\begin{thebibliography}{99}
    	  		
    	  		\small
   	  	    
    	  		\bibitem{AC14}
  	  		Auffinger, A., Chen, W.-K.: Universality of chaos and ultrametricity in mixed p-spin models. {\it Comm. Pure Appl. Math.}, {\bf 69}, no. 11, 2107--2130 (2016)

                \bibitem{C14}
                Chatterjee, S.: Absence of replica symmetry breaking in the random field Ising model. {\it Comm. Math. Phys.}, {\bf 337}, 93–102 (2015)

                \bibitem{CP}
                Chen, W.-K., Panchenko, D.: Some examples of quenched self-averaging of models with Gaussian disorder. {\it Ann. Inst. Henri Poincar\'e Probab. Stat.}, {\bf 53}, no. 1, 243--258 (2017)
%			
			    \bibitem{G98} 
                Ghirlanda, S., Guerra, F.: General properties of overlap probability distributions in disordered spin systems. Towards Parisi ultrametricity. {\it J. Phys. A}, {\bf 31}, no. 46, 9149--9155 (1998)
                
                \bibitem{GT}
                Guerra, F., Toninelli, F.: The thermodynamic limit in mean field spin glass models. {\it Comm. Math. Phys.}, {\bf 230}, no. 1, 71--79 (2002)
               
                \bibitem{MPV}
                M{\'e}zard, M., Parisi, G., Virasoro, M. A.: {Spin glass theory and beyond}. {\bf 9}, {\em World Scientific
                	Lecture Notes in Physics},  World Scientific Publishing Co., Inc., Teaneck, NJ, (1987)
		
    	  		\bibitem{P10}
    	  		Panchenko, D.: The Ghirlanda-Guerra identities for mixed p-spin model. {\it C. R. Acad. Sci. Paris, Ser. I}, {\bf 348}, 189--192 (2010)
%    	  	  				
    	  		\bibitem{P13}
    	  		Panchenko, D.: The Sherrington-Kirkpatrick model. Springer Monographs in Mathematics. Springer, New York (2013)
    	  		
    	  		\bibitem{Par}
    	  	    Parisi, G.: Infinite number of order parameters for spin-glasses. {\it Phys. Rev. Lett.}, {\bf 43}, 1754--1756 (1979)
    	  	    
    	  	    \bibitem{Ro}
    	  	    Roberts, A. W., Verberg, D. E.: Convex Functions. {\it Pure and Applied Mathematics}, {\bf 57}. Academic
    	  	    Press, New York-Landon (1973)

    	  	    
    	  	    	\bibitem{Tal03}
    	  	    	Talagrand, M.: The Parisi formula. {\it Ann. of Math. (2)}, 163(1):221--263 (2006)
                 
    	  	   \bibitem{T11}
    	  	    Talagrand, M.: Mean field models for spin glasses. Ergebnisse der Mathematik und ihrer Grenzgebiete. 3. Folge. A Series of Modern Surveys in Mathematics, {\bf 55}, Springer-Verlag, Berlin (2011)
%    	  	\bibitem{Touchette} Touchette, H.: Equivalence and Nonequivalence of Ensembles: Thermodynamic, Macrostate, and Measure Levels. {\it J. Stat. Phys.}, {\bf 159}, 987--1016 (2015)
	
    	  	\end{thebibliography}
\end{document}